\documentclass[12pt]{article}
\usepackage{arxiv}
\usepackage[utf8]{inputenc} 
\usepackage[T1]{fontenc}    
\usepackage{hyperref}       
\usepackage{url}            
\usepackage{booktabs}       
\usepackage{amsfonts}       
\usepackage{microtype}      
\usepackage{lipsum}

\usepackage{mathrsfs}
\usepackage{amsmath}
\usepackage{amsfonts}
\usepackage{amsthm}
\usepackage{graphicx}
\usepackage{array}
\usepackage{cases}
 \usepackage{microtype}
 
\usepackage[open,openlevel=2,atend]{bookmark}

\usepackage[abbrev,msc-links,backrefs]{amsrefs}
\usepackage{doi}
\usepackage[ruled,vlined,english,linesnumbered]{algorithm2e}

\renewcommand{\PrintDOI}[1]{\doi{#1}}

\newtheorem{thm}{Theorem}[section]

\newtheorem{lem}[thm]{Lemma}
\newtheorem{prop}{Proposition}

\newcommand{\phiinc}[1]{\phi\big({\rm Inc}(#1)\big)}
\newcommand{\varphiinc}[1]{\varphi\big({\rm Inc}(#1)\big)}

\newcommand{\phiprimeinc}[1]{\phi'\big({\rm Inc}(#1)\big)}

\newcommand{\inc}[1]{{\rm Inc}(#1)}

\title{Conflict-free incidence coloring of outer-1-planar graphs}

\author{
  Mengke Qi\\
  School of Mathematics and Statistics\\
  Xidian University\\
  Xi'an 710071, China \\
  \texttt{mkqi@stu.xidian.edu.cn} \\
  \And
 Xin Zhang
 \thanks{
 Supported by the Fundamental
Research Funds for the Central Universities (No.\,QTZX22053) and the National Natural Science Foundation of China (No.\,11871055).
}
 \thanks{Corresponding author.}
 \\
  School of Mathematics and Statistics\\
  Xidian University\\
  Xi'an 710071, China \\
  \texttt{xzhang@xidian.edu.cn} \\
  }

\begin{document}
\maketitle

\begin{abstract}\baselineskip 0.60cm
An incidence of a graph $G$ is a vertex-edge pair $(v,e)$ such that $v$ is incidence with $e$.
A conflict-free incidence coloring of a graph is a coloring of the incidences in such a way that
two incidences $(u,e)$ and $(v,f)$ get distinct colors if and only if they conflict each other, i.e.,\,
(i) $u=v$, (ii) $uv$ is $e$ or $f$, or (iii) there is a vertex $w$ such that $uw=e$ and $vw=f$.
The minimum number of colors used among all conflict-free incidence colorings of a graph is the conflict-free incidence chromatic number. 
A graph is outer-1-planar if it can be drawn in the plane so that vertices are on the outer-boundary and each edge is crossed at most once.
In this paper, we show that the conflict-free incidence chromatic number of an outer-1-planar graph 
with maximum degree $\Delta$ is either $2\Delta$ or $2\Delta+1$ unless the graph is a cycle on three vertices, and moreover, all outer-1-planar graphs with conflict-free incidence chromatic number
$2\Delta$ or $2\Delta+1$ are completely characterized. An efficient algorithm for constructing an optimal conflict-free incidence coloring of a
connected outer-1-planar graph is given.

\vspace{3mm}\noindent \emph{Keywords}: outer-1-planar graph; incidence coloring; combinatorial algorithm; channel assignment problem.
\end{abstract}

\section{Introduction}\label{sec:1}

For groups of geographically separated people who need to keep in continuous voice communication, such as aircraft pilots and air traffic controllers, two-way radios are widely used \cite{enwiki:1038428051}. This motivates us to investigate how to design a 
two-way radio network efficiently and economically. 

In a two-way radio network, each node represents a two-way radio that can 
both transmit and receive radio waves and there is a link between two nodes if and only if they may 
contact each other. Waves can transmit between two linked two-way radios in two different directions simultaneously. For a link $L$  connecting two nodes $N_i$ and $N_j$  in a two-way radio network, it is usually assigned with two channels $C(N_i,N_j)$ and $C(N_j,N_i)$. The former one is used to transmit waves from $N_i$ to $N_j$ and the later one is used to transmit waves from $N_j$ to $N_i$. 
The \textit{associated channel box} $B(N_i)$ of a node $N_i$  in a two-way radio network is a multiset of channels 
$C(N_i,N_j)$ and $C(N_j,N_i)$ such that $N_i$ is linked to $N_j$.
An efficient way to avoid possible interference is to assign channels to links so that every radio receives a rainbow associated channel box (in other words, every two channels in $B(N_i)$ for every node $N_i$ in the network are apart). For the sake of economy, while assigning channels to a two-way radio network, the fewer channels are used, the better. This can be modeled by the conflict-free incidence coloring of graphs.

From now on, we use the language of graph theory and then define conflict-free incidence coloring. 
We consider finite graphs
and use $V(G)$ and $E(G)$ to denote the
vertex set and the edge set of a graph $G$.
The \textit{degree} $d_G(v)$ of a vertex $v$ in a graph $G$ 
is the number of edges incident with $v$ in $G$.
We use $d(v)$ instead of $d_G(v)$ whenever the graph $G$ is clear from the content.
We call $\Delta(G)=\max\{d_G(v)~|~v\in V(G)\}$ and $\delta(G)=\min\{d_G(v)~|~v\in V(G)\}$ the \textit{maximum degree} and the \textit{minimum degree} of a graph $G$. Other undefined notation is referred to  \cite{Diestel2017}.

Let $v$ be a vertex of $G$ and $e$ be an edge incident with $v$. We call the vertex-edge pair $(v,e)$ an \textit{incidence} of $G$.
For an edge $e=uv\in E(G)$, let $\inc{e}=\{(u,e),(v,e)\}$, and for a vertex $v\in V(G)$, let $\inc{v}=\cup_{e\ni v}\inc{e}$.
For a subset $U\subseteq E(G)$, let $\inc{U}=\{\inc{e}~|~e\in U\}$.
Two incidences $(u,e)$ and $(v,f)$ are \textit{conflicting}
if (i) $u=v$,  (ii) $uv$ is $e$ or $f$, or (iii) there is a vertex $w$ such that $uw=e$ and $vw=f$.
In other words, two incidences are conflicting if and only if there is a vertex $w$ such that both of them belong to $\inc{w}$. 

A \textit{conflict-free incidence $k$-coloring} of a graph $G$ is a coloring of the incidences using $k$ colors in such a way that every two conflicting incidences get distinct colors. 
The minimum integer $k$ such that $G$ has a conflict-free incidence $k$-colorable is the \textit{conflict-free incidence chromatic number} of $G$, denoted by $\chi^{c}_i(G)$. 
For a conflict-free incidence coloring $\varphi$ of a graph $G$ and an edge $e=uv\in E(G)$, we use $\varphiinc{e}$ to denote the set $\{\varphi(u,e),\varphi(v,e)\}$. For a subset $U\subseteq E(G)$, let $\varphiinc{U}=\{\varphiinc{e}~|~ e\in U\}$.

We look back into  the channel assignment problem of two-way radio networks and explain why 
the conflict-free incidence coloring of graphs can model it.
Let $G$ be the graph representing the two-way radio network and let $L=N_iN_j$ be an arbitrary link, i.e, $L\in E(G)$. Assigning two channels $C(N_i,N_j)$ and $C(N_j,N_i)$ to $L$ is now equivalent to coloring the incidences $(N_i,L)$ and $(N_j,L)$. The goal of assigning every radio $N_i$ a rainbow associated channel box is translated to coloring the incidences of $G$ so that every two incidences in $\inc{N_i}$ receive distinct colors. This is exactly what we shall do while constructing a 
conflict-free incidence coloring of $G$.

From a theoretical point of view, one may be interested in a fact that the conflict-free incidence coloring relates to 
the \textit{$b$-fold edge-coloring}, which is an assignment of sets of size $b$ to edges of a graph so that adjacent edges receive disjoint sets. An \textit{$(a:b)$-edge-coloring} is a $b$-fold edge coloring out of $a$ available colors. The \textit{$b$-fold chromatic index} $\chi'_b(G)$ is the least integer $a$ such that an $(a:b)$-edge-coloring of $G$ exists. It is not hard to check that $\chi^{c}_i(G)= \chi'_2(G)$ for every graph $G$. However, there are hard problems related to $\chi'_2(G)$, among which the most famous one is the Berge-Fulkerson conjecture \cite{MR294149}, which states that 
every bridgeless cubic graph has a collection of six perfect matchings that together cover every edge exactly twice.  This is equivalent to conjecture that every bridgeless cubic graph $G$ has a $(6:2)$-edge-coloring, i.e., $\chi'_2(G)\leq 6$. This conjecture is still widely open \cite{MR2519156,MR2865657,MR3356176,MR4251536} and was generalized by Seymour \cite{MR532981} to $\gamma$-graphs.

The structure of this paper organizes as follows. In Section \ref{sec:2}, we establish fundamental results for the conflict-free incidence chromatic number of graphs.
In Section  \ref{sec:4}, 
we investigate the conflict-free incidence coloring of outer-1-planar graphs 
by showing that $2\Delta\leq \chi^{c}_i(G)\leq 2\Delta+1$ for outer-1-planar graphs $G$ with maximum degree $\Delta$ unless $G\cong C_3$, and moreover, characterizing  outer-1-planar graphs $G$ with $\chi^{c}_i(G)$ equal to $2\Delta$ or $2\Delta+1$. An efficient algorithm for constructing an optimal conflict-free incidence coloring of a
connected outer-1-planar graph is also given. We end this paper with an interesting open problem relative to the complexity in Section \ref{sec:5}. 

\section{Fundamental results}\label{sec:2}

Let $\chi'(G)$ be the \textit{chromatic index} of $G$, the minimum integer $k$ such that $G$ admits an edge $k$-coloring so that adjacent edges receive distinct colors.
The following is an interesting relationship between $\chi^{c}_i(G)$ and $\chi'(G)$.

\begin{prop} \label{prop:1}
$2\Delta(G)\leq \chi^{c}_i(G)\leq 2\chi'(G)$.
\end{prop}

\begin{proof}
Since $|\inc{v}|=2\Delta(G)$ for a vertex $v$ with maximum degree, $\chi^{c}_i(G)\geq 2\Delta(G)$ for every graph $G$. If $\varphi$ is a proper edge coloring of $G$ using the colors $\{1,2,\ldots,\chi'(G)\}$, then one can construct a conflict-free incidence $2\chi'(G)$-coloring of $G$ such that $\varphi(\inc{e})=\{\varphi(e),\varphi(e)+\chi'(G)\}$ for every edge $e\in E(G)$. It follows that $\chi^{c}_i(G)\leq 2\chi'(G)$.
\end{proof}

The well-known Vizing's theorem (see \cite[p128]{Diestel2017}) states that $\Delta(G)\leq \chi'(G)\leq \Delta(G)+1$ for every simple graph $G$. This divides 
simple graphs into two classes. A simple graph $G$ belongs to \textit{class one} if $\chi'(G)=\Delta(G)$, and belongs to class two\textit{} if $\chi'(G)=\Delta(G)+1$. The following are immediate corollaries of Proposition \ref{prop:1}.

\begin{prop} \label{prop:2}
If $G$ is a class one graph, then $\chi^{c}_i(G)=2\Delta(G)$. \hfill$\square$
\end{prop}

\begin{prop} \label{prop:3}
If $G$ is simple graph, then $\chi^{c}_i(G)\in\{2\Delta(G),2\Delta(G)+1,2\Delta(G)+2\}$. \hfill$\square$
\end{prop}


The well-known K\H{o}nig's theorem (see \cite[p127]{Diestel2017}) states that every bipartite graph is of class 1. So the following is immediate by Proposition \ref{prop:2}.

\begin{thm}
If $G$ is a bipartite graph, then $\chi^{c}_i(G)=2\Delta(G)$. \hfill$\square$
\end{thm}

Now that we have Proposition \ref{prop:2}, it would be worth determining the  conflict-free incidence chromatic number of a certain class of graphs of class two. We first look into  a cycle $C_n$ of length $n$.

\begin{algorithm}[htp]
\BlankLine
\tcc{This algorithm constructs an optimal conflict-free incidence coloring of $C_n$ in linear time.}
\KwIn{The length $n$ of a cycle $C_n$;}
\KwOut{A conflict-free incidence $\chi_i^{c}(C_n)$-coloring $\varphi$ of $C_n$.}
\BlankLine
\tcc{Vertices of $C_n$ are $v_1,v_2,\ldots,v_n$ in this ordering.}
\If{$n=3$\label{li:1}}
{
$\varphiinc{v_1v_2}\gets \{1,2\}$;\\
$\varphiinc{v_2v_3}\gets \{3,4\}$;\\
$\varphiinc{v_3v_1}\gets \{5,6\}$;\\
return;\\
}
$p\gets {\rm the~ quotient~ of~} n~ {\rm divided~ by~ 2}$;\\
$r\gets {\rm the~ remainder~ of~} n~ {\rm divided~ by~ 2}$;\\
\eIf{$r=0$}
{
$v_{2p+1}\gets v_1$;\\
\For{$i=1$ to $2p$}
{
\eIf{$i\equiv 1 \pmod{2}$}
{
$\varphiinc{v_iv_{i+1}}\gets \{1,2\};$\\
}
{$\varphiinc{v_iv_{i+1}}\gets \{3,4\};$\\}
}

}
{
\For{$i=1$ to $2p-2$}
{\eIf{$i\equiv 1 \pmod{2}$}
{$\varphiinc{v_iv_{i+1}}\gets \{1,2\};$\\}
{$\varphiinc{v_iv_{i+1}}\gets \{3,4\};$\\}

}

$\varphiinc{v_{2p-1}v_{2p}}\gets \{1,5\};$\\
$\varphiinc{v_{2p}v_{2p+1}}\gets \{2,3\};$\\
$\varphiinc{v_{2p+1}v_1}\gets \{4,5\};$
}

\caption{\textbf{COLOR-CYCLE}($n$)}
\label{algo:cycle}
\end{algorithm}

\begin{thm} \label{cycle}
\[
   \chi_i^{c}(C_n)= 
   \begin{cases}
        4
        &\text{if } n~is~even,
        \\
        5
        &\text{if } n\geq 5~is~odd,
          \\
        6
        &\text{if } n=3.
        \end{cases}
\]
\end{thm}

\begin{proof}
One can easily see that $C_n$ admits neither a conflict-free incidence $3$-coloring for any integer $n\geq 3$, and nor
a conflict-free incidence $4$-coloring for any odd $n\geq 3$.
Moreover, $C_3$ does not admit a conflict-free incidence $5$-coloring. Hence Algorithm \ref{algo:cycle} outputs a conflict-free incidence coloring of $C_n$ using the least number of colors in linear time and the result follows.
\end{proof}

We now pay attention to the $n$-order complete graph $K_n$. The famous result of Fiorini and Wilson
\cite{MR0543798} states that $K_n$ is of class 1 provided $n$ is even.
Hence Proposition \ref{prop:2} directly imply the following.

\begin{prop} \label{even-kn}
$\chi_i^{c}(K_{2n})=2\Delta(K_{2n})=4n-2$. \hfill$\square$
\end{prop}

Fiorini and Wilson \cite{MR0543798} also showed that
$K_n$ is of class 2 provided $n$ is odd, and thus Proposition \ref{prop:2} cannot be applied to such a $K_n$. Nevertheless, we can  determine the
conflict-free incidence chromatic number of $K_n$ with $n$ being odd from another view of point.

\begin{prop} \label{odd-kn}
If $G$ is the graph derived from $K_{2n+1}$ 
by removing less than $n/2$ edges, then 
$\chi_i^{c}(G)=2\Delta(G)+2=4n+2$. 
\end{prop}

\begin{proof}
We first show that $\chi_i^{c}(G)\ge 4n+2$. 
Suppose for a contradiction that $\varphi$ is a conflict-free incidence $(4n+1)$-coloring 
of $G$.
Since $G$ totally has more than $4n^2+2n-n=(4n+1)n$ 
incidences, there is a color of $\varphi$, say $1$, that has been used at least $n+1$ times.
Since every two strong incidences of a vertex are differently colored, there are $n+1$ vertices of $G$, say $v_1,v_2,\ldots,v_{n+1}$, such that for each $1\leq i\leq n+1$, $\varphi(v_i,v_iu_i)=1$, where $u_i$ is one neighbor of $v_i$. 
Since every two weak incidences of a vertex are also differently colored, each $u_i$ is different from every $u_j$ with $j\neq i$. If $u_i$ coincides with some $v_j$ with $j\neq i$, then $\varphi(v_i,v_iu_i)=\varphi(u_i,u_iu_j)$, a contradiction as $(v_i,v_iu_i)$ conflicts $(u_i,u_iu_j)$.  Hence 
each $u_i$ is different from every $v_j$ with $j\neq i$. It follows that 
$V(G)\supseteq \bigcup_{i=1}^{n+1}\{u_i,v_i\}$ and thus $|V(G)|\geq 2n+2$, a contradiction.
To show the equality, we apply proposition \ref{prop:3} to $G$. It follows that $\chi_i^{c}(G)\le 2\Delta(G)+2=4n+2$, as desired.
\end{proof}

Combining Propositions \ref{even-kn} and \ref{odd-kn} together, we conclude the following.

\begin{thm} \label{complete graph}
\[
   \chi_i^{c}(K_n)= 
   \begin{cases}
        2n-2
        &\text{if } n~is~even,
        \\
        2n
        &\text{if } n~is~odd.
        \end{cases}
\]
\end{thm}

We use the polygon method to construct an optimal conflict-free incidence coloring of $K_n$ by Algorithm \ref{Kn}. 
To analyze the complexity of the algorithm, we need look into its lines \ref{line4} and \ref{line11}. 
If $n$ is even, then for each $1\leq i\leq n-1$, 
$E_i=\{v_{i-j}v_{i+j}|j=1,\dots,\frac{n-2}{2}\}\cup \{v_iv_n\}$ by line \ref{line4}, where the subscripts are taken module $n$ and $v_0$ is recognized as $v_{n-1}$.
If $n$ is odd, then for each $1\leq i\leq n$,  $E_i=\{v_{i-j}v_{i+j+1}|j=0,1,\dots,\frac{n-3}{2}\}$ according to line \ref{line11}, where the subscripts are taken module $n$ and $v_0$ is recognized as $v_{n}$. It follows that the complexity of Algorithm \ref{Kn} is $O((n-1)n/2)=O(n^2)$.

\begin{algorithm}[htp]\label{Kn}
\BlankLine
\tcc{This algorithm constructs an optimal conflict-free incidence coloring of $K_n$ in quadratic time}
\KwIn{The order $n$ of a complete graph $K_n$;}
\KwOut{A conflict-free incidence $\chi_i^{c}(K_n)$-coloring $\varphi$ of $K_n$.}
\BlankLine
\tcc{Vertices of $K_n$ are $v_1,v_2,\ldots,v_n$.}

\eIf{$n\equiv 0 \pmod{2}$}
{$G\gets $ an $(n -1)$-sided regular polygon formed by placing $v_1,v_2,\ldots,v_{n-1}$ on a circle, with $v_n$ at the center of the circle, and connecting every pair of vertices by straight line;\\
\tcc{$G$ now is a special drawing of $K_n$ in the plane.}
\For{$i=1$ to $n-1$}
{
$E_i \gets$ the set of all edges that lie on lines perpendicular to $v_iv_n$ in $G$ along with the edge $v_iv_n$ itself;\label{line4}\\
\For{each edge $e\in E_i$}
{$\varphiinc{e}\gets \{2i-1,2i\}$;\\}
}
}
{$G\gets $ an $n$-sided regular polygon formed by placing $v_1,v_2,\ldots,v_{n}$ on a circle and connecting every pair of vertices by straight line;\\
$v_{n+1}\gets v_1$;\\
\For{$i=1$ to $n$}
{
$E_i \gets$ the set of all edges that lie on lines parallel to $v_iv_{i+1}$ in $G$ along with the edge $v_iv_{i+1}$ itself;\label{line11}\\
\For{each edge $e\in E_i$}
{$\varphiinc{e}\gets \{2i-1,2i\}$;\\}
}
}
\caption{\textbf{COLOR-COMPLETE-GRAPH}($n$)}
\end{algorithm}

\section{Outer-1-planar graphs}\label{sec:4}

In this section we determine the conflict-free incidence chromatic numbers of outer-$1$-planar graphs, a subclass of planar partial $3$-trees \cite{MR3162015}, which serve many applications ranging from network reliability to machine learning.
Formally speaking, a graph is \textit{outer-1-planar} if it can be drawn in the plane so that vertices are on the outer-boundary and each edge is crossed at most once.
The notion of outer-1-planarity was first introduced by Eggleton \cite{MR846198} and outer-1-planar graphs are also known as
\emph{outerplanar graphs with edge crossing number one} \cite{MR846198} and \emph{pseudo-outerplanar graphs} \cites{MR2945171,MR3084275,MR3203677}.
The coloring of outer-1-planar graphs were investigated by many authors including \cites{MR2945171,MR3442550,MR3084275,MR3203677,MR4338069,MR4218010,MR4144341,MR3704822,MR3967165,QZ22,zbMATH07367782}.

The most popular result on the edge coloring of planar graphs is that planar graphs with maximum degree at least 7 is of class one \cite{MR1804346,MR1866396}.
Since there exist class two planar graphs with maximum degree $\Delta$ for each $\Delta\leq 5$, the remaining problem is to determine whether every planar graph with maximum degree 6 is of class one, and this is still quite open (see survey \cite{MR3898375}). Therefore, investigating the edge coloring of subclasses of planar graphs is natural and interesting.
Fiorini \cite{MR366724} showed that every outerplanar graph is of class one if and only if it is not an odd cycle, and this conclusion had  been generalized to the class of series-parallel graphs by Juvan, Mohar, and Thomas \cite{MR1728012}. Zhang, Liu, and Wu \cite{MR2945171} showed that outer-1-planar graphs with maximum degree at least 4 are of class one. The chromatic indexes of outer-1-planar graphs with maximum degree at most 3 was completely determined by Zhang \cite{MR3442550}. 

\begin{figure}
    \centering
    \includegraphics[width=14cm]{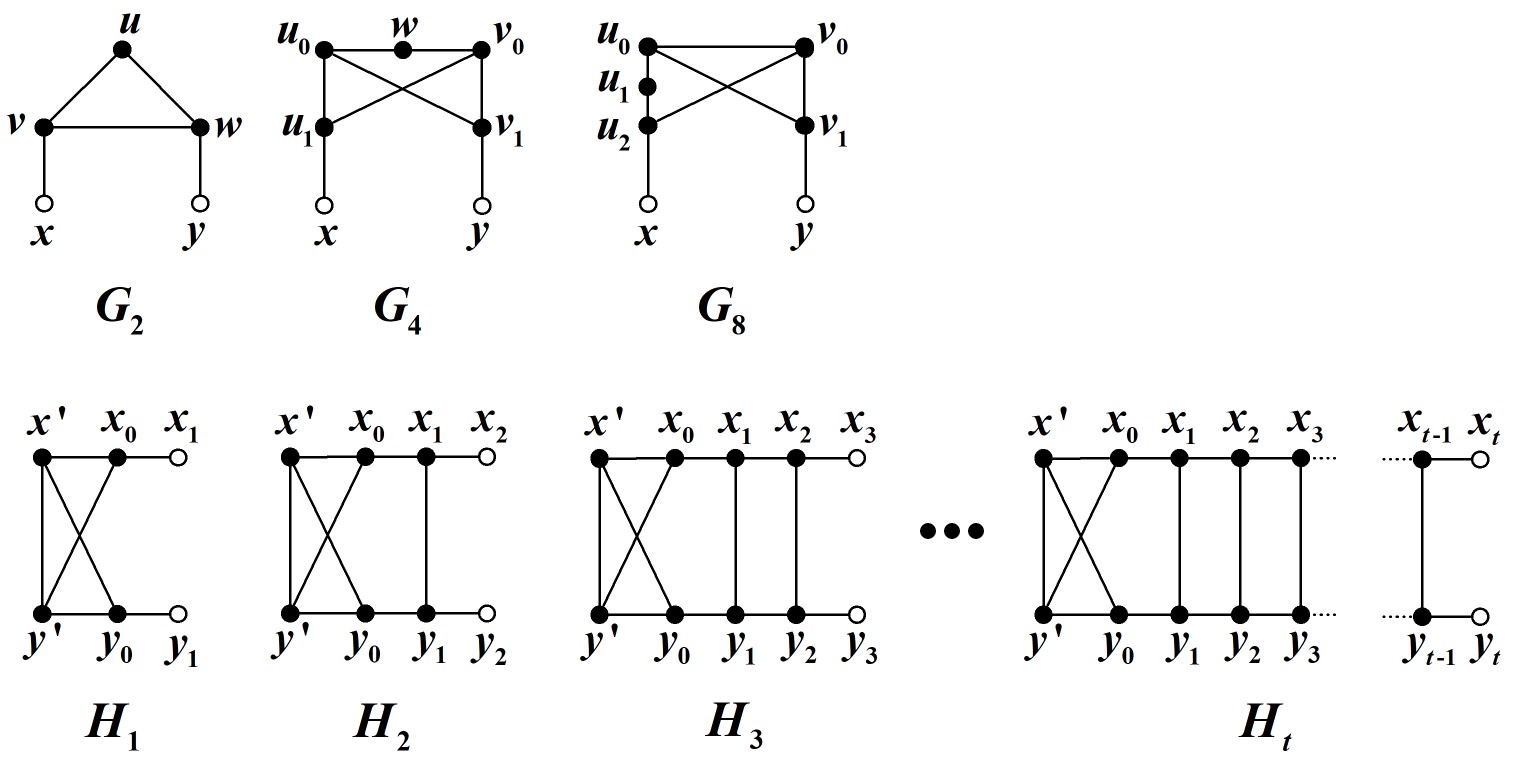}
    \caption{The configurations $G_2,G_4,G_8$ and $H_t$}
    \label{fig:operation1}
\end{figure}

We restate Zhang's definition \cite{MR3442550} as follows. Let $G_2,G_4,G_8,$ and $H_t$ be configurations defined by Figure \ref{fig:operation1}. For any solid vertex $v$ of a configuration and any graph $G$ containing such a configuration, the degree of $v$ in $G$ is exactly the number of edges that are incident with $v$ in the picture.

A graph belongs to the class $\mathcal{P}$, if it is isomorphic to $K_4^+$ (equal to $K_4$ with one edge subdivided) or derived from a graph $G\in \mathcal{P}$ by one of the following operations:
\begin{description}
  \item[$\boldsymbol{G\sqcup_z G_t}$ with $\boldsymbol{t=2,4,8}$] \label{o-1}
  remove a vertex $z$ of degree two from $G$, and then
  paste a copy of $G_2$, or $G_4$, or $G_8$ on the current graph accordingly, by identifying $x$ and $y$ with $z_1$ and $z_2$, respectively, where $z_1$ and $z_2$ are the neighbors of $z$ (see Figure \ref{fig:operation3} for an example);
  \begin{figure}[htp]
    \centering
    \includegraphics[width=6cm]{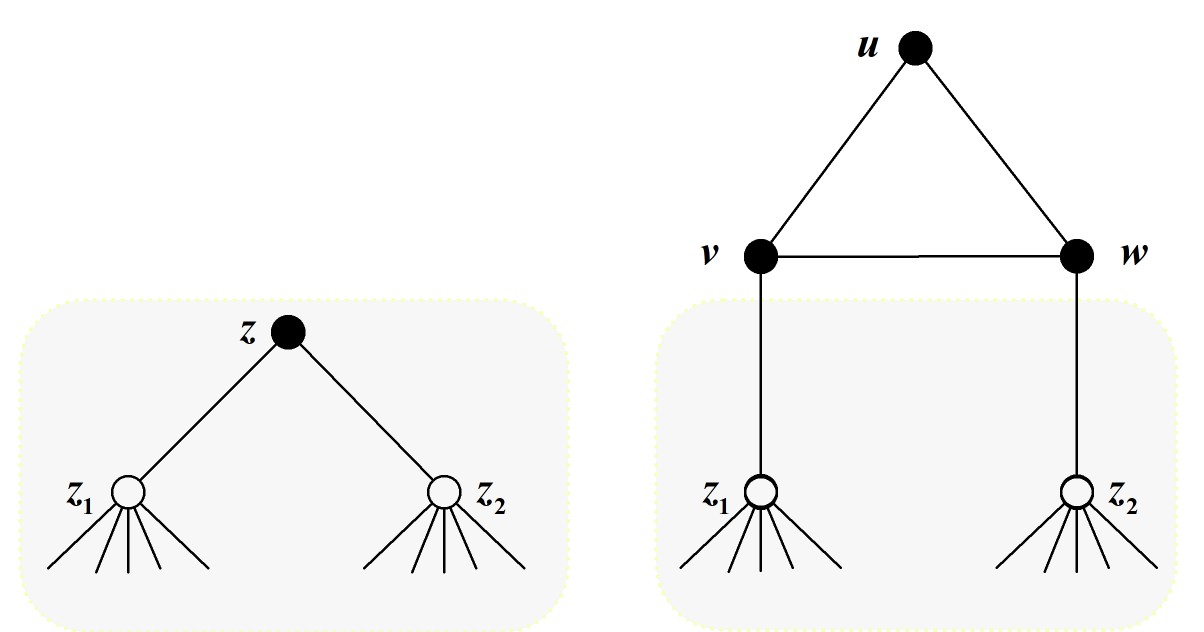}
    \caption{The graph on the left shows $G$ and the one on the right shows $G\sqcup_z G_2$}
    \label{fig:operation3}
\end{figure}
  \item[$\boldsymbol{G\vee_{z_1z_2} H_t}$ with $\boldsymbol{t\geq 1}$] \label{o-2}
  remove an edge $z_1z_2$ from $G$, and then paste a copy of $H_t$ on the current graph by identifying $x_t$ and $y_t$ with $z_1$ and $z_2$, respectively  (see Figure \ref{fig:operation4} for an example).
  \begin{figure}[htp]
    \centering
    \includegraphics[width=6cm]{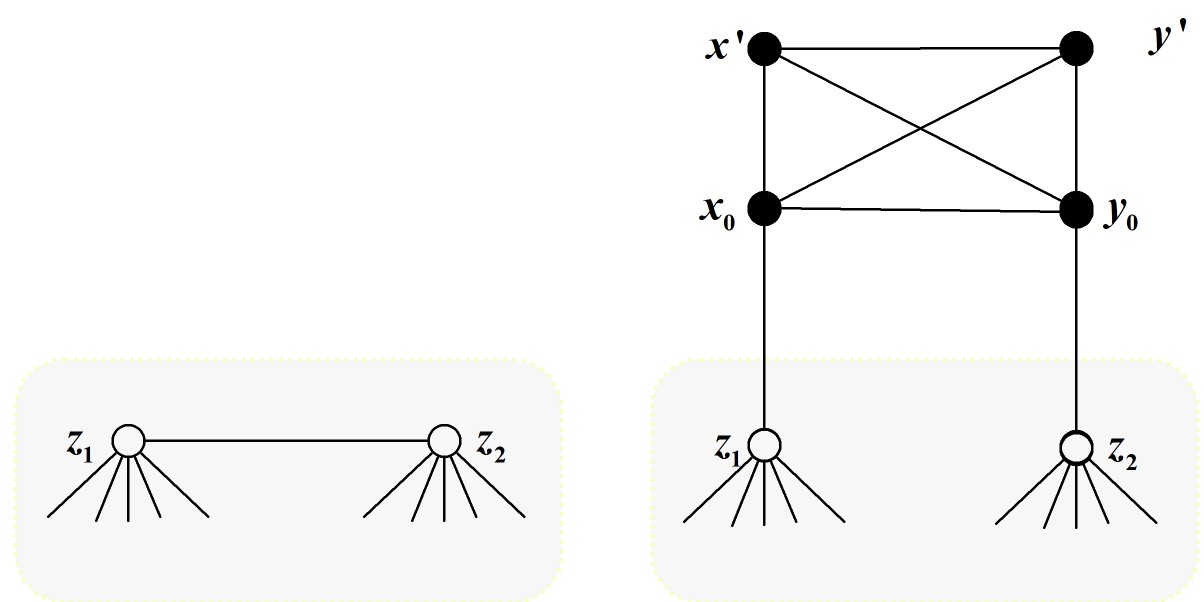}
    \caption{The graph on the left shows $G$ and the one on the right shows $G\vee_{z_1z_2}H_1$}
    \label{fig:operation4}
\end{figure}
\end{description}

Let $\mathcal{P}^+$ be the class of connected outer-1-planar graphs with maximum degree 3 that contains some graph in $\mathcal{P}$ as a subgraph.
Now we summarize the result of Zhang \cite{MR3442550} and Zhang, Liu, and Wu \cite{MR2945171} as follows. 

\begin{thm} \label{edge-col}
\[
   \chi'(G)= 
   \begin{cases}
        \Delta(G)
        &\text{if } G\not\in \mathcal{P}^+~and~ G ~is ~not ~an ~odd ~cycle,
        \\
        \Delta(G)+1
        &\text{otherwise},
        \end{cases}
\]if $G$ is a connected outer-$1$-planar graph. \hfill$\square$
\end{thm}

\noindent \textbf{Remark on Theorem \ref{edge-col}:} \textit{Zhang \cite{MR3442550} claimed that every connected outer-1-planar graph with maximum degree 3 is of class one if and only if $G\not\in \mathcal{P}$. However, this statement is incorrect. Indeed, Zhang showed that every graph in  $\mathcal{P}$ is of class two. This further implies that every outer-1-planar graph with maximum degree 3 that contains some graph in $\mathcal{P}$ is of class two. In other words, every graph in  $\mathcal{P}^+$ is of class two. Using the same proof of Theorem 3.3 in \cite{MR3442550}, one can show that if $G$ is a connected outer-1-planar graph with maximum degree 3  not in $\mathcal{P}^+$ then it is of class one (note that the minimal counterexample to this statement is 2-connected and thus Zhang's original proof works now).
Conclusively, every connected outer-1-planar graph with maximum degree 3 is of class one if and only if $G\not\in \mathcal{P}^+$. Combining this with the result of Zhang, Liu, and Wu \cite{MR2945171} that every outer-1-planar graph with maximum degree at least 4 is of class one, we have Theorem \ref{edge-col}.}

The following is an immediate corollary of Theorem \ref{edge-col} and Proposition \ref{prop:1}.

\begin{thm}\label{thm:notp}
If $G$ is a connected outer-1-planar graph such that $G\not\in \mathcal{P}^+$ and $G$ is not an odd cycle, then $\chi_i^{c}(G)=2\Delta(G)$. \hfill$\square$
\end{thm}

The next goal of this section is to prove $\chi_i^{c}(G)=2\Delta(G)+1$ if $G\in \mathcal{P}^+$ or $G$ is an odd cycle unless $G\cong C_3$. Theorem \ref{cycle} supposes this conclusion while 
$G$ is an odd cycle of length at least 5. Hence in the following we assume that $G\in \mathcal{P}^+$. 
Note that  $K_4^{+}$  is the smallest graph  (in terms of the order) in $\mathcal{P}^+$.
Now we prove $\chi_i^{c}(G)=7$ for every graph $G\in \mathcal{P}^+$
by a series of lemmas.

\begin{lem} \label{k4}
$\chi_i^{c}(K_4^{+})=7$.
\end{lem}

\begin{proof}
Figure \ref{fig:operation2} shows a conflict-free incidence 7-colorable of $K_4^+$, so it is sufficient to show that 6 colors are not enough to create a conflict-free incidence coloring of $K_4^+$.
\begin{figure}[htp]
    \centering
    \includegraphics[width=3cm]{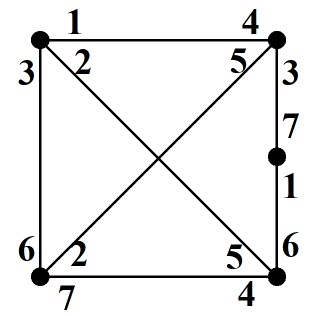}
    \caption{A conflict-free incidence 7-colorable of $K_4^+$}
    \label{fig:operation2}
\end{figure}

 Suppose for a contradiction that $\varphi$ is a conflict-free incidence 6-coloring of $K_4^+$.
 Since $K_4^{+}$ has $7$ edges and $14$ incidences, there is a color, say $1$, such that
 $\varphi(x_1,x_1x_1')=\varphi(x_2,x_2x'_2)=\varphi(x_3,x_3x'_3)=1$. If $x_i=x_j$ or $x'_i=x'_j$ or $x_i=x'_j$ for some $1\leq i<j\leq 3$, then $(x_i,x_ix'_i)$ and $(x_j,x_jx'_j)$ are conflicting and thus they cannot in a same color.
 Hence $|\{x_1,x_2,x_3,x'_1,x'_2,x'_3\}|=6$, contradicting the fact that $|K^+_4|=5$.
 \end{proof}

From now on, if we say coloring a graph or a configuration we mean coloring its incidences so that every two conflicting ones receive distinct colors.

\begin{lem}\label{l2}
If the configuration $G_2$ is colored with $6$ colors under $\varphi$, then  $\varphiinc{vx}\cap  \varphiinc{wy}=\emptyset$.
\end{lem}

\begin{proof}
If $\varphi$ is a conflict-free incidence $6$-coloring of $G_2$, then $\varphi (u,uv), \varphi (v,uv), \varphi (u,uw), \varphi (w,uw),\varphi (v,vw)$ and $\varphi (w,vw)$ are pairwise distinct, so we assume, without loss of generality, that they are $1,2,3,4,5$, and $6$, respectively. This forces that $\varphiinc{vx}=\{3,4\}$ and $\varphiinc{wy}=\{1,2\}$, as desired.
\end{proof}

\begin{lem}\label{l3}
If the configuration $G_4$ is colored with $6$ colors under $\varphi$, then $\varphiinc{u_1x}\cap  \varphiinc{v_1y}=\emptyset$.
\end{lem}

\begin{proof}
If $\varphi$ is a conflict-free incidence $6$-coloring of $G_4$, we have three cases: $\varphiinc{u_1x}=\varphiinc{v_1y}$, or $\varphiinc{u_1x}\cap \varphiinc{v_1y}=\emptyset$, or $|\varphiinc{u_1x}\cap \varphiinc{v_1y}|=1$. If $\varphiinc{u_1x}\cap \varphiinc{v_1y}=\emptyset$, then we win. So it is sufficient to show contradictions for another two cases.
Without loss of generality, we assume $\varphiinc{u_1x}=\{1,2\}$,
$\varphiinc{u_1v_0}=\{3,4\}$, and $\varphiinc{u_0u_1}=\{5,6\}$.

\vspace{2mm}\textbf{Case 1.} \textit{ $\varphiinc{u_1x}=\varphiinc{v_1y}$.}

Now $\varphiinc{v_1y}\cup \varphiinc{(u_0u_1}=\{1,2,5,6\}$ and
 $\varphiinc{v_1y}\cup \varphiinc{u_1v_0}=\{1,2,3,4\}$
 forces $\varphiinc{u_0v_1}=\{3,4\}$ and
$\varphiinc{v_0v_1}=\{5,6\}$, respectively.
It follows
$\varphiinc{u_0u_1,u_0v_1}=\varphiinc{u_1v_0,v_0v_1}=\{3,4,5,6\}$ and thus
$\varphiinc{u_0w}=\varphiinc{v_0w}=\{1,2\}$, which is impossible.

\vspace{2mm}\textbf{Case 2.} \textit{ $|\varphiinc{u_1x}\cap \varphiinc{v_1y}|=1$.}

Assume, by symmetry, that $\varphiinc{v_1y}=\{1,a\}$, where $a\in\{3,4\}$. It follows that
$\varphiinc{v_1y}\cup \varphiinc{u_0u_1}=\{1,a,5,6\}$, forcing
$\varphiinc{u_0v_1}=\{2,b\},~b\in\{3,4\}\setminus \{a\}$.
Now
$\varphiinc{u_0u_1,u_0v_1}=\{2,b,5,6\}$
and
 $\varphiinc{v_1y}\cup \varphiinc{u_0v_1}=
 \{1,2,3,4\}$,
which implies
$\varphiinc{u_0w}=\{1,a\}$
and
$\varphiinc{v_0v_1}=\{5,6\}$, respectively. It follows that
 $\varphiinc{u_1v_0,v_0v_1,u_0w}=\{1,3,4,5,6\}$
 and thus $\inc{wv_0}$ have to be colored with 2, which is impossible.
\end{proof}

\begin{lem}\label{l4}
If the configuration $G_8$ is colored with $6$ colors under $\varphi$, then $\varphiinc{u_2x}\cap  \varphiinc{v_1y}=\emptyset$.
\end{lem}

\begin{proof}
If $\varphi$ is a conflict-free incidence $6$-coloring of $G_8$,
 we have three cases: $\varphiinc{u_2x}=\varphiinc{v_1y}$, or $\varphiinc{u_2x}\cap \varphiinc{v_1y}=\emptyset$, or $|\varphiinc{u_2x}\cap \varphiinc{v_1y}|=1$. If $\varphiinc{u_2x}\cap \varphiinc{v_1y}=\emptyset$, then we win. So it is sufficient to show contradictions for another two cases.
Without loss of generality, we assume $\varphiinc{v_1y}=\{1,2\}$,
$\varphiinc{u_0v_1}=\{3,4\}$, and $\varphiinc{v_0v_1}=\{5,6\}$.

\vspace{2mm}\textbf{Case 1.} \textit{ $\varphiinc{u_2x}=\varphiinc{v_1y}$.}

Now $\varphiinc{u_2x}\cup \varphiinc{v_0v_1}=\{1,2,5,6\}$ and
 $\varphiinc{u_0v_1,v_0v_1}=\{3,4,5,6\}$
 forces $\varphiinc{u_2v_0}=\{3,4\}$ and
$\varphiinc{u_0v_0}=\{1,2\}$, respectively. 
It follows
$\varphiinc{u_0v_0,u_0v_1}=\varphiinc{u_2x}\cup \varphiinc{u_2v_0}=\{1,2,3,4\}$
and thus
 $ \varphiinc{u_0u_1}=\varphiinc{u_1u_2}=\{5,6\}$,
 which is impossible.

\vspace{2mm}\textbf{Case 2.} \textit{ $|\varphiinc{u_2x}\cap \varphiinc{v_1y}|=1$.}

Assume, by symmetry, that $\varphiinc{u_2x}=\{1,a\}$, where $a\in\{5,6\}$. It follows that
$\varphiinc{u_0v_1}\cup \varphiinc{v_0v_1}=\{3,4,5,6\}$ and $\varphiinc{v_0v_1}=\{5,6\}$, forcing
$\varphiinc{u_0v_0}=\{1,2\}$ and $\varphiinc{u_2v_0}=\{3,4\}$.
Now
$\varphiinc{u_0v_0,u_0v_1}=\{1,2,3,4\}$
which implies
$\varphiinc{u_0u_1}=\{5,6\}$. It follows that
 $\varphiinc{u_0u_1,u_2x,u_2v_0}=\{1,3,4,5,6\}$
 and thus $\inc{u_1u_2}$ have to be colored with 2, which is impossible.
\end{proof}

\begin{lem}\label{l5}
If the configuration $H_t$ with some $t\geq 1$   is colored with $6$ colors under $\varphi$,
then $\varphiinc{x_{t-1}x_t}=\varphiinc{y_{t-1}y_t}$.
\end{lem}

\begin{proof}
We prove it by induction on $t$.
If $\varphi$ is a conflict-free incidence $6$-coloring of $H_1$, then we assume, without loss of generality,
   $\varphi(x',x'y'), \varphi(y',x'y'), \varphi(x',x'y_0), \varphi(y_0,x'y_0),
 \varphi(x',x'x_0)$, and $\varphi(x_0,x'x_0)$ are $1,2,3,4,5$, and $6$, respectively.
 Since $\varphiinc{x'y',x'x_0}=\{1,2,5,6\}$ and $\varphiinc{x'y',x'y_0}=\{1,2,3,4\}$, we have $\varphiinc{x_0y'}=\{3,4\}$ and $\varphiinc{y'y_0}=\{5,6\}$, which imply $\varphiinc{x_0x_1}=\varphiinc{y_0y_1}=\{1,2\}$. This completes the proof of the base case.
 Now suppose that the lemma holds for $H_{t-1}$ with some $t\geq 2$ and prove that it also holds for $H_{t}$. By the induction hypothesis, $\varphiinc{x_{t-2}x_{t-1}}=\varphiinc{y_{t-2}y_{t-1}}$. This implies $\varphiinc{x_{t-1}x_{t}}=\{1,2,3,4,5,6\}\setminus \{\varphiinc{x_{t-2}x_{t-1}}\cup \varphiinc{x_{t-1}y_{t-1}}\}$ and $\varphiinc{y_{t-1}y_{t}}=\{1,2,3,4,5,6\}\setminus \{\varphiinc{y_{t-2}y_{t-1}}\cup \varphiinc{x_{t-1}y_{t-1}}\}$, and thus $\varphiinc{x_{t-1}x_t}=\varphiinc{y_{t-1}y_t}$, as desired.
 \end{proof}

\begin{lem}\label{l6}
If $\varphi$ is a partial incidence coloring of the configuration $G_2$ such that $\varphiinc{vx}\cap \varphiinc{wy}=\emptyset,$
 then $\varphi$ can be extended to a conflict-free incidence $6$-coloring of the configuration $G_2$.
\end{lem}

\begin{proof}
 Suppose $\varphiinc{vx}=\{1,2\}$ and $\varphiinc{wy}=\{3,4\}$. It is easy to see that we can extend $\varphi$ to a conflict-free incidence $6$-coloring of $G_2$ by coloring $\inc{uv,uw,vw}$ so that $\varphiinc{uv}=\{3,4\}$, $\varphiinc{uw}=\{1,2\}$, and $\varphiinc{vw}=\{5,6\}$.
\end{proof}

\begin{lem}\label{l7}
If $\varphi$ is a partial incidence coloring of the configuration $G_4$ such that
$\varphiinc{u_1x}\cap \varphiinc{v_1y}=\emptyset,$
 then $\varphi$ can be extended to a conflict-free incidence $6$-coloring of the configuration $G_4$.
\end{lem}

\begin{proof}
 Suppose $\varphiinc{u_1x}=\{1,2\}$ and $\varphiinc{v_1y}=\{3,4\}$. It is easy to see that we can extend $\varphi$ to a conflict-free incidence $6$-coloring of $G_4$ by coloring  $\inc{u_0v_1, v_0w, u_0w, u_1v_0, u_0u_1,v_0v_1}$ so that $\varphiinc{u_0v_1}=\varphiinc{v_0w}=\{1,2\}$, $\varphiinc{u_0w}=\varphiinc{u_1v_0}=\{3,4\}$, and $\varphiinc{u_0u_1}=\varphiinc{v_0v_1}=\{5,6\}$.
\end{proof}

\begin{lem}\label{l8}
If $\varphi$ is a partial incidence coloring of the configuration $G_8$ such that
$\varphiinc{u_2x}\cap \varphiinc{v_1y}=\emptyset,$
then $\varphi$ can be extended to a conflict-free incidence $6$-coloring of the configuration $G_8$.
\end{lem}

\begin{proof}

 Suppose $\varphiinc{u_2x}=\{1,2\}$ and $\varphiinc{v_1y}=\{3,4\}$. We can extend $\varphi$ to a conflict-free incidence $6$-coloring of $G_4$ by coloring the incidences on $v_0v_1, u_0u_1, u_0v_0, u_1u_2, u_0v_1$, and $u_2v_0$ so that $\varphiinc{v_0v_1}=\varphiinc{u_0u_1}=\{1,2\}$, $\varphiinc{u_0v_0}=\varphiinc{u_1u_2}=\{3,4\}$, and $\varphiinc{u_0v_1}=\varphiinc{u_2v_0}=\{5,6\}$.

\end{proof}

\begin{lem}\label{l9}
If $\varphi$ is a partial incidence coloring of the configuration $H_t$ with some $t\geq 1$ such that
$\varphiinc{x_{t-1}x_t}=\varphiinc{y_{t-1}y_t}$,
then $\varphi$ can be extended to a conflict-free incidence $6$-coloring of the configuration $H_t$.
\end{lem}

\begin{proof}
We prove it by induction on $t$.
If $\varphi$ is a partial incidence coloring of the configuration $H_1$ such that
$\varphiinc{x_{0}x_1}=\varphiinc{y_{0}y_1}=\{1,2\}$, then $\varphi$ can be extended to a conflict-free incidence $6$-coloring of $H_t$ by coloring $\inc{x'y', x'y_0, x_0y', x'x_0, y'y_0}$ so that $\varphiinc{x'y'}=\{1,2\}$, $\varphiinc{x'y_0}=\varphiinc{x_0y'}=\{3,4\}$, and $\varphiinc{x'x_0}=\varphiinc{y'y_0}=\{5,6\}$.
This completes the proof of the base case.
 Now suppose that the lemma holds for $H_{t-1}$ with some $t\geq 2$ and prove that it also holds for $H_{t}$.
Assume, without loss of generality, that  $\varphiinc{x_{t-1}x_t}=\varphiinc{y_{t-1}y_t}=\{1,2\}$. We extend $\varphi$ by coloring $\inc{x_{t-2}x_{t-1},y_{t-2}y_{t-1},x_{t-1}y_{t-1}}$ so that $\varphiinc{x_{t-2}x_{t-1}}=\varphiinc{y_{t-2}y_{t-1}}=\{3,4\}$ and $\varphiinc{x_{t-1}y_{t-1}}=\{5,6\}$.
This constructs a partial incidence coloring of the configuration $H_{t-1}=H_t-\{x_{t-1}y_{t-1},x_{t-1}x_t,y_{t-1}y_t\}$ such that
$\varphiinc{x_{t-2}x_{t-1}}=\varphiinc{y_{t-2}y_{t-1}}$. Since any incidence of $\inc{x_{t-1}y_{t-1},x_{t-1}x_t,y_{t-1}y_t}$ is conflict-free to any
incidence of $\inc{H_{t-1}}$, by the induction hypothesis, the extended $\varphi$ can be further extended to a conflict-free incidence $6$-coloring of the configuration $H_t$.
\end{proof}

\begin{prop}\label{thm:p}
If $G\in \mathcal{P}$, then $\chi_i^{c}(G)=7$.
\end{prop}

\begin{proof}
We proceed by induction on $|G|$. Since the smallest graph in $\mathcal{P}$ is $K_4^+$, and $\chi_i^{c}(K_4^+)=7$ by Lemma \ref{k4}, the proof of the base case has been done. Now assume $|G|>5$. By the construction of $\mathcal{P}$, we meet four cases.
Here and elsewhere, once $G$ contains a configuration as shown in Figure \ref{fig:operation1}, we use the same labelling of any vertex appearing on the configuration as the one marked in the corresponding picture.

\vspace{2mm}\textbf{Case 1.} \textit{
There is a graph $G'\in \mathcal{P}$ and a degree 2 vertex $z$ of $G'$ such that $G=G'\sqcup_z G_2$ (or $G=G'\sqcup_z G_4$, or $G=G'\sqcup_z G_8$, respectively).}

\vspace{2mm}
By the induction hypothesis, $\chi_i^{c}(G')=7$. Let $z_1,z_2$ be two neighbors of $z$ in $G'$ and let $\varphi$ be a conflict-free incidence $7$-coloring of $G'$. Clearly, $\varphiinc{zz_1}\cap  \varphiinc{zz_2}=\emptyset$. We construct a conflict-free incidence $7$-coloring $\phi$ of $G$ as follows.
Let $\phiinc{vx}=\varphiinc{zz_1}$ and $\phiinc{wy}=\varphiinc{zz_2}$ (or $\phiinc{u_1x}=\varphiinc{zz_1}$ and $\phiinc{v_1y}=\varphiinc{zz_2}$ ,  or  $\phiinc{u_2x}=\varphiinc{zz_1}$ and $\phiinc{v_1y}=\varphiinc{zz_2}$ , respectively). This makes a partial incidence coloring of the configuration $G_2$ (or $G_4$, or $G_8$, respectively) such that $\varphiinc{vx}\cap \varphiinc{wy}=\emptyset$ (or $\varphiinc{u_1x}\cap \varphiinc{v_1y}=\emptyset$, or $\varphiinc{u_2x}\cap \varphiinc{v_1y}=\emptyset$, respectively). By Lemma \ref{l6} (or Lemma \ref{l7}, or Lemma \ref{l8}, respectively), $\varphi$ can be extended to a conflict-free incidence $7$-coloring of the configuration $G_2$ (or $G_4$, or $G_8$, respectively)  and thus any two conflicting incidences of $I(E(G)\setminus E(G'))$  receive distinct colors.
Now for every edge $e\in E(G)\cap E(G')$, let $\phiinc{e}=\varphiinc{e}$. This completes a $7$-coloring of the incidences of $G$ and it is easy to check that this coloring is  conflict-free.

On the other hand, we show that $G$ admits no conflict-free incidence $6$-coloring. Suppose, for a contradiction, that $\phi$ is a conflict-free incidence $6$-coloring of $G$. By Lemma \ref{l2} (or Lemma \ref{l3}, or Lemma \ref{l4}, respectively), $\phiinc{vx}\cap \phiinc{wy}=\emptyset$ (or $\phiinc{u_1x}\cap \phiinc{v_1y}=\emptyset$, or $\phiinc{u_2x}\cap \phiinc{v_1y}=\emptyset$, respectively).
This makes us possible to construct a conflict-free incidence $6$-coloring $\varphi$ of $G'$
by setting $\varphiinc{zz_1}=\phiinc{vx}$, $\varphiinc{zz_2}=\phiinc{wy}$, (or $\varphiinc{zz_1}=\phiinc{u_1x}$, $\varphiinc{zz_2}=\phiinc{v_1y}$, or $\varphiinc{zz_1}=\phiinc{u_2x}$, $\varphiinc{zz_2}=\phiinc{v_1y}$, respectively) and $\varphiinc{e}=\phiinc{e}$ for every edge $e\in  E(G')\setminus E(G)$.
This is a contradiction.

\vspace{2mm}\textbf{Case 2.} \textit{
There is a graph $G'\in \mathcal{P}$ and an edge $z_1z_2$ of $G'$ such that $G=G'\vee_{z_1z_2} H_i$.}

By the induction hypothesis, $\chi_i^{c}(G')=7$. Let $\varphi$ be a conflict-free incidence $7$-coloring of $G'$.
We construct a conflict-free incidence $7$-coloring $\phi$ of $G$ as follows.
Let $\phiinc{x_{i-1}x_i}=\phiinc{y_{i-1}y_i}=\varphiinc{z_1z_2}$. This makes a partial incidence coloring of the configuration $H_i$ such that $\phiinc{x_{i-1}x_i}=\phiinc{y_{i-1}y_i}.$
By Lemma \ref{l9}, $\phi$ can be extended to a conflict-free incidence $7$-coloring of the configuration $H_i$.
Now for every edge $e\in E(G)\cap E(G')$, let $\phiinc{e}=\varphiinc{e}$. This completes a $7$-coloring of the incidences of $G$ and it is easy to check that this coloring is  conflict-free.

On the other hand, we show that $G$ admits no conflict-free incidence $6$-coloring. Suppose, for a contradiction, that $\phi$ is a conflict-free incidence $6$-coloring of $G$. By Lemma \ref{l5}, $\phiinc{x_{i-1}x_i}=\phiinc{y_{i-1}y_i}$.
This makes us possible to construct a conflict-free incidence $6$-coloring $\varphi$ of $G'$
by setting $\varphiinc{z_1z_2}=\phiinc{x_{i-1}x_i}$ and $\varphiinc{e}=\phiinc{e}$ for every edge $e\in  E(G')\setminus E(G)$.
This is a contradiction.
\end{proof}

Algorithm \ref{algo:color-p} summarises the idea of proving Theorem \ref{thm:p}, showing how we can construct a conflict-free incidence $7$-coloring of a graph in $\mathcal{P}$ efficiently. Now we are ready to prove a more general result as follows.

\begin{algorithm}[htp]
\BlankLine
\KwIn{A graph $G\in \mathcal{P}$;}
\KwOut{A conflict-free incidence $7$-coloring $\varphi$ of $G$.}
\BlankLine
$i\gets 0;$\\
$G_0\gets G;$\\

\While{$G_i\not\cong K_4^-$}  
{
\eIf{there is a graph $G'\in \mathcal{P}$ with a degree 2 vertex $z$ such that $G_i=G'\sqcup_z G_t$ for some $t\in \{2,4,8\}$;}
{$G_{i+1}\gets G'$;\\
${\rm sign}_{i}\gets t$;\\}
{Find a graph $G'\in \mathcal{P}$ with an edge $z_1z_2$ such that $G_i=G'\vee_{z_1z_2} H_t$ for some integer $t$;\\
$G_{i+1}\gets G'$;\\
${\rm sign}_{i}\gets 0$;\\}
$i\gets i+1;$\\
}
\tcc{We obtain a series $G_0,G_1,\ldots,G_i$ of graphs in $\mathcal{P}$ where $G_0=G$ and $G_i=K_4^-$.}
Construct a conflict-free $7$-coloring $\varphi_i$ of $G_i$ by Lemma \ref{k4};\\
\For{$j=i-1$ to $0$}
{
Extend $\varphi_{j+1}$ to a conflict-free $7$-coloring $\varphi_j$ of $G_j$ by Lemma \ref{l6}, \ref{l7}, \ref{l8}, or \ref{l9} whenever ${\rm sign}_{j}$ equals to 2, 4, 8, or 0, respectively;\\
}
$\varphi\gets \varphi_0$;
\caption{\textbf{COLOR-CLASS-P}($G$)}
\label{algo:color-p}
\end{algorithm}

\begin{algorithm}[htp]
\BlankLine
\KwIn{A graph $G\in \mathcal{P^+}$;}
\KwOut{A conflict-free incidence $7$-coloring $\varphi$ of $G$.}
\BlankLine

\eIf{$G\in \mathcal{P}$}
{{\rm \textbf{COLOR-CLASS-P}}($G$);\label{aa}\\
\tcc{The coloring outputted by line \ref{aa} is denoted by $\varphi$.}}
{Find a subgraph $H\in \mathcal{P}$ of $G$ with a vertex $u$ that has exactly two neighbors $v$ and $w$ in $H$;\\
$H' \gets$ the graph with vertex set $V(G)\setminus (V(H)\setminus \{u\})$ and edge set $(E(G)\setminus E(H))\cup \{ux\}$;\\
$x\gets$ the unique neighbor of $u$ in $H'$;\\
\eIf{$H'\in \mathcal{P^+}$}
{
\textbf{COLOR-CLASS-P-PLUS}($H'$);\label{a}\\
\tcc{The coloring outputted by line \ref{a} is denoted by $\phi'$.}
}
{
Find a proper edge $3$-coloring $\varphi'$ of $H'$ by Theorem \ref{edge-col};\\
\For{each edge $e\in H'$}
{$\phiprimeinc{e}\gets \{\varphi'(e),\varphi'(e)+3\};$}
}
\textbf{COLOR-CLASS-P}($H$);\label{b}\\
\tcc{The coloring outputted by line \ref{b} is denoted by $\phi$.}
Exchange (if necessary) the colors of $\phi$ so that $\phiinc{uv}$, $\phiinc{uw}$, and $\phiprimeinc{ux}$ are pairwise disjoint;\\
$\varphi \gets $ the coloring obtained via combing $\phi'$ with $\phi$;
}
\caption{\textbf{COLOR-CLASS-P-PLUS}($G$)}
\label{alg: color-p+}
\end{algorithm}

\begin{thm}\label{thm:p+}
If $G\in \mathcal{P}^+$, then $\chi_i^{c}(G)=7$.
\end{thm}

\begin{proof}
We proceed by induction on $|G|$. Note that the base case is supported by Lemma \ref{k4}.
By the definition of $\mathcal{P}$, every graph in $\mathcal{P}$ has exactly one vertex of degree 2, besides which all vertices are of degree 3. 
By Proposition \ref{thm:p}, we assume $G\in \mathcal{P}^+\setminus \mathcal{P}$. 

Suppose that $G$ contains a graph $H\in \mathcal{P}$ as a proper subgraph. Let $u$ be the unique vertex of degree 2 of $H$ and let $v$ and $w$ be the two neighbors of $u$ in $H$. Since $\Delta(G)\leq 3$ and $G$ is connected, the degree of $u$ in $G$ must be 3. Let $x$ be the third neighbor of $u$ in $G$. Since every vertex in $V(H)\setminus \{u\}$ has degree 3 in $H$ (and thus in $G$), $u$ is a cut-vertex of $G$. 

Let $H'$ be the subgraph of $G$ containing $u$ such that $V(H')\cap V(H)=\{u\}$ and $V(H')\cup V(H)=V(G)$.
Since $u$ has degree 1 in $H'$, $H'$ is not an odd cycle. Therefore, if $H'\in \mathcal{P}^+$, then $\chi_i^{c}(H')=7$ by the induction hypothesis, and if $H'\not\in \mathcal{P}^+$, then $\chi'(H')=\Delta(H')\leq 3$ by Theorem \ref{edge-col} and thus $\chi_i^{c}(H')\leq 6$ by Proposition \ref{prop:1}. In each case, there is a conflict-free incidence $7$-coloring $\phi'$ of $H'$.

Since $H\in \mathcal{P}$, there is a conflict-free incidence $7$-coloring $\phi$ of $H$ by Proposition \ref{thm:p}. We permute (if necessary) the colors of $\phi$ so that $\phiinc{uv}$, $\phiinc{uw}$, and $\phiprimeinc{ux}$ are pairwise disjoint, and then obtain a 
conflict-free incidence $7$-coloring of $G$ by combining $\phi'$ with $\phi$. This implies $\chi_i^{c}(G)\leq 7$.

On the other hand, $\chi_i^{c}(G)\geq \chi_i^{c}(H)=7$. Hence $\chi_i^{c}(G)=7$.
\end{proof}

Algorithm \ref{alg: color-p+} shows the idea of constructing a conflict-free incidence $7$-coloring of a give graph in $\mathcal{P^+}$.
Now that we have Theorems \ref{cycle}, \ref{thm:notp}, and \ref{thm:p+}, the conflict-free incidence chromatic number of connected outer-1-planar graphs (and thus all outer-1-planar graphs) can be completely determined by Theorem \ref{out-1-planar}.
Algorithm \ref{algo:color-o1p} shows an approach to efficiently construct a conflict-free incidence $\chi_i(G)$-coloring $\varphi$ of a connected out-1-planar graph $G$.

\begin{thm} \label{out-1-planar}
\[
   \chi^{c}_i(G)= 
   \begin{cases}
         6
        &\text{if } G\cong C_3,
        \\2\Delta(G)
        &\text{if } G\not\in \mathcal{P}^+~and~G ~is ~not ~an ~odd ~cycle,
       \\
       
        2\Delta(G)+1
        &\text{otherwise}
        \end{cases}
\]
for every connected outer-$1$-planar graph $G$.
\end{thm}

\begin{algorithm}[htp]
\BlankLine
\KwIn{A connected out-1-planar graph $G$;}
\KwOut{A conflict-free incidence $\chi_i(G)$-coloring $\varphi$ of $G$.}
\BlankLine
\tcc{This algorithm constructs an optimal conflict-free incidence coloring of a connected outer-1-planar graph $G$.}
\eIf{$G$ is a cycle}
{\textbf{COLOR-CYCLE}($|G|$);}
{
\eIf{$G\in \mathcal{P^+}$ }
{\textbf{COLOR-CLASS-P-PLUS}($G$);}
{
Find a proper edge $\Delta(G)$-coloring $\phi$ of $G$ by Theorem \ref{edge-col};\\
\For{each edge $e\in G$}
{$\varphiinc{e}\gets \{\phi(e),\phi(e)+\Delta(G)\};$}
}

}

\caption{\textbf{COLOR-O1P}($G$)}
\label{algo:color-o1p}
\end{algorithm}

\section{Open problem}\label{sec:5}

To end this paper, we leave an open problem relative to the complexity of the conflict-free incidence coloring.
As one can know from Proposition \ref{prop:3} that $\chi^{c}_i(G)\in\{2\Delta(G),2\Delta(G)+1,2\Delta(G)+2\}$
for every simple graph $G$, an interesting problem is to investigate the complexity of the following question.

 \textsc{\textbf{Conflict-free incidence coloring Problem (CFICP)}}\\
\indent Input: A graph $G$ and a positive integer $k$.\\
\indent Question: Is there a conflict-free incidence $k$-coloring of $G$?\\
We conjecture that \textsc{\textbf{CFICP}} is NP-Complete.

\bibliographystyle{abbrv}
\bibliography{ref}

\end{document}